\documentclass[letter, 12pt]{amsart}

\usepackage[T1]{fontenc}
\usepackage[margin=1in]{geometry}
\usepackage{amscd}
\usepackage{amsmath}
\usepackage{amsfonts}
\usepackage{amssymb}
\usepackage{amsthm}

\usepackage{amsmath}
\usepackage{mathrsfs}
\usepackage{mathtools}
\usepackage{enumitem}

\usepackage{arydshln}
\usepackage{fancyhdr}
\usepackage{verbatim}

\usepackage{color}

\usepackage[colorlinks=true, linkcolor=blue, citecolor=blue,
pagebackref=true]{hyperref}

\usepackage{mathtools} \usepackage{etoolbox}
\patchcmd{\section}{\scshape}{\bfseries}{}{} \makeatletter
\renewcommand{\@secnumfont}{\bfseries} \makeatother
\theoremstyle{definition}

\theoremstyle{plain} \newtheorem{theorem}{Theorem}[section]

\theoremstyle{plain} \newtheorem{lemma}[theorem]{Lemma}

\theoremstyle{plain} \newtheorem{proposition}[theorem]{Proposition}

\theoremstyle{plain} 

\theoremstyle{plain} \newtheorem{corollary}[theorem]{Corollary}

\theoremstyle{remark} \newtheorem*{remark}{Remark}

\theoremstyle{definition} \newtheorem{definition}[theorem]{Definition}

\theoremstyle{definition} \newtheorem*{definition*}{Definition}

\theoremstyle{definition} 

\theoremstyle{definition} \newtheorem{notation}[theorem]{Notation}

\makeatletter \renewenvironment{proof}[1][\proofname]
{\par\pushQED{\qed}\normalfont\topsep6\p@\@plus6\p@\relax\trivlist\item[\hskip\labelsep\bfseries#1\@addpunct{.}]\ignorespaces}{\popQED\endtrivlist\@endpefalse}
\makeatother

\newcommand{\N}{\mathbb{N}}
\newcommand{\Q}{\mathbb{Q}}
\newcommand{\R}{\mathbb{R}}
\newcommand{\Z}{\mathbb{Z}}
\newcommand{\Rpos}{\R_{> 0}}

\newcommand{\Zpos}{\Z_{> 0}}

\newcommand{\Zp}[1][p]{\Z_{#1}}
\newcommand{\Qp}[1][p]{\Q_{#1}}

\newcommand{\RmodZ}{\R/\Z}

\DeclarePairedDelimiter{\parens}{\lparen}{\rparen}
\DeclarePairedDelimiter{\brackets}{\lbrack}{\rbrack}
\DeclarePairedDelimiter{\set}{\lbrace}{\rbrace}
\DeclarePairedDelimiter{\abs}{\lvert}{\rvert}

\newcommand{\setofall}{\left\{}
\newcommand{\suchthat}{:}
\newcommand{\setend}{\right\}}

\newcommand{\norm}[1]{\abs*{ #1 } }
\newcommand{\pnorm}[2][p]{ \norm{#2}_{#1} }

\newcommand{\pappby}[3][p]{W_{#1}\parens*{#2, #3}}
\newcommand{\appby}[2]{\pappby[\empty]{#1}{#2}}

\newcommand{\m}[2][\empty]{\mu_{#1}\parens*{#2}}

\newcommand{\eps}{\varepsilon}

\newcommand{\F}[2][B]{\mathscr{F}_{#2}^{#1}}
\newcommand{\psum}{\sideset{}{'}\sum}
\newcommand{\da}{\,d\alpha}
\newcommand{\bphi}[1][B]{\varphi^{#1}}
\newcommand{\man}[2]{\max\set{ \abs{#1}, \abs{#2} } }

\usepackage{xparse}
\NewDocumentCommand{\dan}{
  O{a}
  D[]{n}
  }{%
  \delta_{(#1,#2), \psi}
}

\newcommand{\Haus}[1][]{\mathcal{H}^{#1}}
\newcommand{\haus}[2][f]{\Haus[#1]\parens*{#2} }
\newcommand{\hdim}[1]{\mathrm{dim}_{\Haus}\parens*{#1}}

\begin{document}

% Header Info.
\title[Neighborhoods in different places]{Khintchine's Theorem with rationals coming from neighborhoods in different places}
\author[A.~P.~Oliveira]{Andre P. Oliveira}
\address{Wesleyan University, Middletown CT, USA} 
\email{aoliveira@wesleyan.edu}
\date{\today}

\begin{abstract}
  The Duffin--Schaeffer Conjecture answers a question on how well one can approximate irrationals by rational numbers in reduced form (an imposed condition) where the accuracy of the approximation depends on the rational number.
  It can be viewed as an analogue to Khintchine's Theorem with the added restriction of only allowing rationals in reduced form.
  Other conditions such as numerator or denominator a prime, a square-free integer, or an element of a particular arithmetic progression, etc. have also been imposed and analogues of Khintchine's Theorem studied.
  We prove versions of Khintchine's Theorem where the rational numbers are sourced from a ball in some completion of $\Q$ (i.e. Euclidean or $p$-adic), while the approximations are carried out in a distinct second completion.
  Finally, by using a mass transference principle for Hausdorff measures, we are able to extend our results to their corresponding analogues with Haar measures replaced by Hausdorff measures, thereby establishing an analogue of Jarn\'ik's Theorem.
\end{abstract}

\thispagestyle{empty}

\maketitle

\tableofcontents

\section{Introduction}

Loosely speaking, the field of metric Diophantine approximations is concerned with approximating elements of a metric space by a specified dense subset, where the accuracy of approximation is measured against the particular element in the dense subset.
Classically, this is posed as: given a real number $x$, how small can $\abs{x - \frac{a}{n}}$ be, as a function of $n$, with $\frac{a}{n} \in \Q$. 
One potential answer to this question is to investigate the convergence or divergence of a particular sum concerning the rate of approximation:
\begin{theorem}[Khintchine's Theorem, \cite{khintchine}]\label{thm:khintchine}
  If $\psi : \N \to \Rpos$ is a monotonic function, then \[
    \lambda( W(\psi) ) = \begin{cases}
      0 & \sum\limits_{n=1}^\infty n \psi(n) < \infty \\
      1 & \sum\limits_{n=1}^\infty n \psi(n) =   \infty, \\
  \end{cases}
  \]
  where $\lambda$ is Lebesgue measure and \[ W(\psi) = \setofall x \in [0,1] \suchthat \abs*{x - \frac{a}{n} } < \psi(n), a \in \Z, a \le n, \emph{ for i.m. } n \in \N \setend, \]
  where i.m. stands for ``infinitely many.''
\end{theorem}

An analogous statement is known for higher dimensions, $\R^d$, where one simultaneously approximates $d$ real numbers by $d$ rationals with a common denominator.
In the one-dimensional case, it is known that one cannot remove the monotonicity assumption on $\psi$ and obtain the same result; a counter-example may be found in \cite{duffin-schaeffer}.
The recently settled Duffin--Schaeffer Conjecture (see \cite{Maynard-Koukoulopoulos}) can be viewed as a Khintchine-like Theorem rationals that are \emph{restricted} to reduced form.

Other restrictions have also been studied.
For example, an analogue of Khintchine's theorem where one restricts to rationals with prime numerator and denominator was established by Harman \cite{harman_prime_num_and_den}. A discussion of this as well as the more general setup where numerator and denominator come from more general infinite subsets of $\N$ is covered in \cite[Chapter 6]{harman_book}.
Similar higher-dimensional ``primality constraints'' have been investigated by Baier and Ghosh in \cite{baier-ghosh_1,baier-ghosh_2}.
Akin to the ``primality constraints'' work has also been done where the numerator and denominator come from predetermined arithmetic progressions. A reasonably complete theory of this may be found in \cite{faustin-arith_prog}.
Recently there have also been probabilistic approaches, where one considers random numerators for a given denominator, by Ram\'irez \cite{felipe-random}.

The results listed thus far pertain to $\Q$ (or a subset of $\Q$) sitting densely in the metric space, $\R$. What about other metric spaces; for example, $\Qp$, the $p$-adic numbers?

Recall that for any prime $p$ we may define a norm, $\pnorm{\cdot}$, called the \emph{$p$-adic norm} on $\Q$ by \[
  \pnorm{ \frac{a}{b} } = \pnorm{ \frac{a' p^k}{b'} } = p^{-k},
\] where $p \nmid a', b'$.
The \emph{$p$-adic numbers} are the completion of $\Q$ with respect to this $p$-adic norm.
This space is a locally compact and totally disconnected field. It has as its ring of integers \[
  \Zp := \setofall \alpha \in \Qp \suchthat \pnorm{\alpha} \le 1 \setend,
\]
a compact group under addition. Thus $\Zp$ may be endowed with a unique Haar probability measure, which will be denoted by $\m[p]{\cdot}$.

Since $\Qp$ is the completion of $\Q$ under this norm, the rationals sit inside of $\Qp$ as a proper dense subset and are thus a viable set to do Diophantine approximations.
In 1940, Mahler \cite{mahler} studied $p$-adic Diophantine approximations from a continued fraction perspective.
This was followed by Jarn\'ik, in 1945, with an analogue to Khintchine's Theorem in the $p$-adic setting \cite{jarnik}.
In her 1955 thesis, Lutz showed that an analogous higher dimensional Khintchine's Theorem holds for systems of linear forms \cite{lutz}. In a sense, our Theorem \ref{thm:p_to_p} Case 2, where $p_1 = \infty$ and $p_2 = p$, falls into this category.

These two directions in Diophantine approximations (i.e. in $\R$ or $\Qp$) are not disjoint. Palmer in \cite{palmer} shows multiple analogues of theorems in Diophantine approximations (and metric number theory) hold in the setting of ``diagonal Diophantine approximation.'' That is, approximating elements of \[ 
  X = \R \times \Q_{p_1} \times \Q_{p_2} \times \cdots \times \Q_{p_k},
\] by elements of a diagonal embedding of $\Q$ into the space $X$, where $p_i$ are distinct primes for each $i = 1,2,\ldots, k$.

In a similar spirit, Haynes showed \cite[Theorem 4]{haynes} that a higher-dimensional Khintchine's Theorem holds in the setting of simultaneous mixed approximations. That is, approximating elements of
\[
  \R^l \times \Q_{p_1} \times \Q_{p_2} \times \cdots \times \Q_{p_k},
\]
where $l$ and $k$ are nonnegative integers with $l + k > 1$, and $p_1, p_2, \cdots, p_k$ are (not necessarily distinct) primes, by rational vectors ($\Q^{l+k}$).
Further, in the same paper, Haynes shows that if a certain variance method from probability theory (quasi-independence on average) can be used to solve the $p$-adic Duffin--Schaeffer conjecture for even one prime $p$, then almost the entire classical Duffin--Schaeffer conjecture would follow.
Conversely, if the variance method can be used to prove the classical conjecture, then the $p$-adic conjecture is true for all primes.

\subsection{The Setup}

In this article, we take the ``restricted rationals'' approach. The goal is to show when one can approximate the elements of one completion of $\Q$ (a target) by rationals coming from a ball in a distinct other completion of $\Q$ (a source), where the ability to approximate, or not approximate, elements depends solely on the approximating function and not the completions. 

Informally our result says the following. Consider two distinct completions of $\Q$: $\widehat\Q$ and $\widetilde\Q$.
Take a ball, $\widehat{B} \subseteq \widehat{\Q}$, in one completion and consider only the rationals, $Q = \Q \cap \widehat B$, in that ball.
It then follows that almost every (resp. almost no) element of $\widetilde\Q$ may be approximated well (i.e. with respect to a function $\psi$) by rationals in $Q$ if the sum \[
  \sum_{n=1}^\infty n \psi(n)
\]
diverges (resp. converges).

To be more precise, instead of considering the full completions of $\Q$ (i.e. $\R$ and $\Qp$) as our sources we focus on natural compact subsets associated with both. Specifically, we consider approximations between $\Zp$ and $[0,1]$, $\RmodZ$ and $\Zp$, as well as between two distinct $p$-adic completions.

\subsection{Definitions and Result}

As a viability check, it is known that a ball in one completion of $\Q$ is dense in any other completion of $\Q$.
Moreover, since a ball in $\RmodZ$ lifts to balls in $\R$, it follows that the the rationals contained in these equivalence classes, in $\RmodZ$, are also dense in $\Zp$.

With that, we consider the following definitions of $\psi$-approximability in the $p$-adic, Euclidean, and mixed cases.

\begin{definition}
  Let $p_1, p_2$ be two distinct places. Fix a ball $B_{p_1} \subseteq \Zp[p_1]$.
  A $p_2$-adic number $\alpha \in \Zp[p_2]$ is \textbf{$\psi$-approximable by rationals in $B_{p_1}$} if there exist infinitely many rationals $\frac{a}{n} \in B_{p_1}$ such that
  \begin{equation}\label{eq:def_approx}
    \pnorm[p_2]{ \alpha - \frac{a}{n} } < \psi(\man{a}{n});
  \end{equation}
  where $\Zp[\infty] = \RmodZ$
  
  The \textbf{$p_2$-adic $\psi$-approximables by rationals in $B_{p_1}$} are denoted by
  \begin{align*}
    \pappby[p_2]{\psi}{B_{p_1}} := \setofall \alpha \in \Zp[p_2] \suchthat \pnorm[p_2]{ \alpha - \frac{a}{n} } < \psi(\man{a}{n}) \text{ for i.m. } \frac{a}{n} \in B_{p_1} \setend.
  \end{align*}
\end{definition}

Before continuing we clarify a few points of potential issues with notation.
Given a ball $B \subset \RmodZ$ and a coset $\left[ \frac{a}{n} \right] \in B$, 
a rational $\frac{a}{n} \in \Q$, will at times be written $\frac{a}{n} \in B \subseteq \RmodZ$ as shorthand for \[ \brackets*{\frac{a}{n}} := \frac{a}{n} + \Z \in B \subseteq \RmodZ, \] since $B \subseteq \RmodZ$ is by definition the set of equivalence classes of $\frac{a}{n} \in \Q$.
However, when we say the coset $\alpha \in \Zp[\infty] = \RmodZ$ is $\psi$-approximable (e.g. $p_1 = p$, $p_2 = \infty$), we mean that the coset representative of $\alpha$ lying in the unit interval $[0,1)$ is $\psi$-approximable.
That is, when $\Zp[\infty] = \RmodZ$ our set $\pappby[p_2]{\psi}{B_{p_1}} \subseteq \RmodZ$ should be viewed as a subset of the interval $[0,1)$, the canonical fundamental domain of $\RmodZ$.
Then, the inequality in \eqref{eq:def_approx}, becomes the more familiar \[
  \pnorm[\infty]{ \alpha - \frac{a}{n} } < \psi(n),
\] since $a \le n$.
In fact, this viewpoint is at the crux of the proof in the mixed case.

\begin{remark}
  While there have been numerous approaches to studying Diophantine approximations in the $p$-adics, the approach above, where the ``height'' of $\frac{a}{n}$ is $\man{a}{n}$, is taken from the convention used in \cite{jarnik,lutz}.
\end{remark}

With our notation and definitions set, we state the main result of our paper:

\begin{theorem}\label{thm:p_to_p}
  Given two distinct places, $p_1, p_2$, fix a ball $B_{p_1} \subseteq \Zp[p_1]$; we take $\Z_\infty = \RmodZ$.
  If $\psi : \Rpos \to \Rpos$ is monotonically decreasing and for all $s > 1$, there exists $c > 0$, so that $\psi(sn) > c \psi(n)$ for all sufficiently large $n$,  then \[
    \m[p_2]{ \pappby[p_2]{\psi}{B_{p_1}} } = \begin{cases}
      1 & \sum\limits_{n=1}^\infty n \psi(n) = \infty \\
      0 & \sum\limits_{n=1}^\infty n \psi(n) < \infty.
    \end{cases}
  \]
\end{theorem}

\begin{remark}
  The proofs of the above yield slightly stronger results; the rationals used in the approximations are reduced (i.e. relatively prime numerator and denominator). To emphasize the connection with Khintchine's Theorem we have stated the results in this form.
\end{remark}

We divide the theorem into three cases and prove them separately in Sections \ref{sec:p_to_r}, \ref{sec:r_to_p}, and \ref{sec:p_to_p}, respectively: 
\begin{enumerate}[font=\bfseries,label=Case \arabic*:]
  \item $p_1$ is prime, $p_2 = \infty$,
  \item $p_1 = \infty$, $p_2$ is prime,
  \item $p_1, p_2$ are primes.
\end{enumerate}

\begin{remark}
  As we will see shortly the proofs of Case 2 and Case 3 are, unsurprisingly, very similar; They both reduce to Diophantine approximations occurring in $\Z_{p_2}$ with a counting problem involving rationals coming from \emph{some} source. The nature of these sources poses a technical difficulty in the counting, but the gist of both arguments is the same.
  To that end, and for a clearer presentation, we prove Case 2 on its own and then clarify the changes needed for Case 3 afterward.
\end{remark}

\begin{remark}
  It would be interesting to further develop our results in the contexts of Haynes and Palmer's work.
  That is, simultaneous and diagonal approximations, respectively.
  While this is not answered in this paper, we hope to extend our work in these and other higher-dimensional contexts in the future.
\end{remark}

Lastly, we include a generalization of Khintchine's Theorem to a broader class of Hausdorff measures. See Section \ref{sec:hausdorff_generalization} for the definitions and details.

\begin{theorem}[Jarn\'ik Theorem]\label{thm:jarnik}
  Suppose $\psi : \N \to \Rpos$ is monotonically decreasing and
  for all $s > 1$, there exists some $c > 0$ such that $\psi(sb) > c \psi(b)$.
  Let $f$ be a dimension function.
  Then \[
    \haus{ \pappby[p_2]{\psi}{B_{p_1}} } = \begin{cases}
      \haus{\Zp[p_2]} & \sum\limits_{n=1}^\infty n f(\psi(n)) = \infty \\
      0 & \sum\limits_{n=1}^\infty n f(\psi(n)) < \infty,
    \end{cases}
  \]
  where $\Zp[p_1]$ is the set of $p_1$-adic integers, when $p_1$ is prime, and is the set $\Z_\infty = \RmodZ$, otherwise.
\end{theorem}

From this an analogue of the classic Jarn\'ik--Besicovitch Theorem also follows.
Higher dimensional versions of a $p$-adic Jarn\'ik--Besicovitch Theorem have previously been studied in \cite{abercrombie,dickinson-dodson-yuan}.
\begin{corollary}[Jarn\'ik--Besicovitch Theorem,]\label{cor:jarnik-besicovitch}
  Given two distinct places, $p_1, p_2$, fix a ball $B_{p_1} \subseteq \Zp[p_1]$; we take $\Z_\infty = \RmodZ$. Then, for $\tau \ge 2$,
  \[
  \hdim{\pappby[p_2]{\tau}{B_{p_1}}} = \frac{2}{\tau}
  \]
  where $\pappby[p_2]{\tau}{B_{p_1}} := \pappby[p_2]{q \mapsto q^{-\tau}}{B_{p_1}}$.
\end{corollary}

\begin{remark}
  As we will see, the proof of this theorem is stronger than merely a dimension result.
  The method of proof tells us that $\pappby[p_2]{\tau}{B_{p_1}}$ has infinite Hausdorff $s$-measure at the critical exponent, $s = \frac{2}{\tau}$.
\end{remark}

\section{Some useful Lemmas}

We recall here some classical and elementary results that will be of use later.
The following Lemma is used throughout the paper to consider $p$-free series.
\begin{lemma}
  Suppose $\psi : \Rpos \to \Rpos$ is a monotonically decreasing function. Then,
  \[
    \sum_{n=1}^\infty n \psi(n) = \infty \qquad \text{if and only if} \qquad \sum_{\substack{n=1 \\ p \nmid n}}^\infty n \psi(n) = \infty
  \]
\end{lemma}

The following is a well known and classical result used in metric Diophantine approximations. It may be found in \cite[Lemma 2.1]{harman_book}.
\begin{lemma}[Cassel's Lemma]\label{lem:cassels}
  Let $\mathscr{I}_k \subseteq \R$ be a sequence of intervals and $\mathscr{E}_k \subseteq \R$ be a sequence of measurable sets such that, for some $\delta \in (0,1)$, \[
    \mathscr{E}_k \subseteq \mathscr{I}_k, \qquad \m{\mathscr{E}_k} \ge \delta \m{\mathscr{I}_k}, \quad \m{\mathscr{I}_k} \to 0.
  \]
  Then the set of points which belong to infinitely many of the $\mathscr{I}_k$ has the same measure as the set of points which belong to infinitely many of the $\mathscr{E}_k$.
\end{lemma}

We also recall the following Lemma of Lutz which will be useful in showing a Zero-One Law in Case 2 of Theorem \ref{thm:p_to_p}. In Section \ref{sec:p_to_p} we will make use of a new modified version of this Lemma to prove Case 3 of Theorem \ref{thm:p_to_p}.
\begin{lemma}\cite[Lemma 1.1]{lutz}\label{lem:lutz1.1}
  Suppose $A \subseteq \Zp$ is a positive measure set. Given any $\eps > 0$, there exists a positive integer $c_1 = c_1(\eps)$, such that the set \[
    \setofall \alpha = \alpha' + z \suchthat \alpha' \in A, z \in \Z, \abs{z} < c_1 \setend
  \]
  has measure at least $1 - \eps$.
\end{lemma}

\section{Proof of Theorem \ref{thm:p_to_p}, Case 1}\label{sec:p_to_r}

Fix some ball $B_p \subseteq \Zp$. As in the classical case, the convergence half of the Theorem may be shown directly by a Borel--Cantelli covering argument. More directly, one has \[
  \appby{\psi}{B_p} \subseteq W(\psi),
\] where $W(\psi)$ is again the classical collection of $\psi$-approximables. Thus, with the sum converging we have $\m{W(\psi)} = 0$, by the classical Khintchine Theorem \ref{thm:khintchine} and so $\m{\appby{\psi}{B_p}} = 0$ as well.

Towards the divergence case, one can further restrict to only considering source balls centered at the origin in $\Zp$. Then, applying Cassel's lemma yields the desired result.

\begin{proposition}[Reduction to neighborhoods of zero]
  Let $\psi : \Rpos \to \Rpos$ be a monotonically decreasing function. Suppose that for all $s > 1$, there exists some $c > 0$ such that $\psi(sb) > c \psi(b)$, for all sufficiently large $b$, then \[
    \m{ \appby{\psi}{ B_p } } = \m{ \appby{\psi}{ B_p^0 } },
  \] where $B_p^0$ is the ball $B_p$ translated to be centered at the origin.
\end{proposition}
\begin{proof}
  Recall that every ball in $\Zp$ can be expressed as \[ 
    B_p = B_p\parens*{ \frac{r}{s}, p^{-k} }= \setofall \alpha \in \Zp \suchthat \pnorm{\frac{r}{s} - \alpha} < p^{-k} \setend,
  \]
  for some $\frac{r}{s} \in Q$ and $k \ge 0$; that is, it is the $p$-adic ball of radius $p^{-k}$ centered at the rational $\frac{r}{s}$.
  We then have \[
    B_p \cap \Q = \setofall \frac{r}{s} + \frac{a p^k}{b} \suchthat a,b \in \Z, p \nmid b \setend
  \]
  and so \[
    B_p^0 \cap \Q = \setofall \frac{a p^k}{b} \suchthat a,b \in \Z, p \nmid b \setend.
  \]
  
  Given $x \in \appby{\psi}{B_p}$ we have infinitely many solutions to \[
    \norm{ x - \frac{a}{n} } < \psi(n),
  \] with $\frac{a}{n} \in B_p$. Since every rational in $B_p \cap \Q$ is of the form $\frac{r}{s} + \frac{a p^k}{b}$, with $p \nmid b$ the left hand side of the above can be rewritten as \[
    \norm{ x - \parens*{\frac{r}{s} + \frac{a p^k}{b} } }.
  \]
  Then, since $\psi(n)$ is monotonic we have $\psi(sb) \le \psi(b)$ for any $s > 1$, and by rearranging the terms we have infinitely many solutions to \[
    \norm{ \parens*{x - \frac{r}{s} } -  \frac{a p^k}{b} } < \psi(b);
  \] that is, $\parens*{x - \frac{r}{s}} \in \appby{\psi}{B_p^0}$. It then follows that \[
    \m{\appby{\psi}{B_p}} \le \m{\appby{\psi}{B_p^0}}
  \] since translations (i.e. the map $x \mapsto \parens*{x - \frac{r}{s} }$) are measure-preserving and invertible.
  
  Now, taking $x \in \appby{\psi}{B_p^0}$ we have infinitely many solutions, $a \in \Z, b \in \Z \setminus p\Z$, to 
  \begin{align*}
    \norm{ x - \frac{a p^k}{b} } &< \psi(b) \\
    \norm{ \parens*{x + \frac{r}{s}} - \parens*{ \frac{r}{s} + \frac{a p^k}{b} } } &< \psi(b)
  \end{align*}
  and since $\psi(b) < \frac{1}{c} \psi(sb)$ for all sufficiently large $b$ we have infinitely many solutions to \[
    \norm{ \parens*{x + \frac{r}{s}} - \parens*{ \frac{r}{s} + \frac{a p^k}{b} } } < \psi(b) < \frac{1}{c} \psi(sb);
  \] that is, $\parens*{x + \frac{r}{s}} \in \appby{\frac{1}{c} \psi}{B_p}$.
  Again, because this map is a translation we have \[
    \m{\appby{\psi}{B_p^0}} \le \m{\appby{\frac{1}{c} \psi }{B_p}}.
  \]
  So, we have the following chain of inequalities \[
    \m{\appby{\psi}{B_p}}
    \le
    \m{\appby{\psi}{B_p^0}} 
    \le
     \m{\appby{\frac{1}{c} \psi}{B_p}}.
  \] However, by Cassel's Lemma, Lemma \ref{lem:cassels} above, we have \[
  \m{\appby{\psi}{B_p}} =   \m{\appby{\frac{1}{c} \psi}{B_p}}
  \]
  and so \[
  \m{\appby{\psi}{B_p}} = \m{\appby{\psi}{B_p^0}}, 
  \] as needed.  
\end{proof}

\begin{comment}
The argument to apply this Lemma is standard; observe that $ \norm{ x - \frac{a}{n} } < \psi(n) $ is a real ball about $\frac{a}{n}$ of radius $\psi(n)$, which is measurable. Also note that from the monotonicity of $\psi(n)$ we have $c < 1$.
Define \[
  \mathscr{E}_n = \setofall x \in [0,1] \suchthat \norm{x - \frac{a}{n}} < \psi(n) \setend \quad \text{ and } \quad \mathscr{I}_n = \setofall x \in [0,1] \suchthat \norm{x - \frac{a}{n}} < \frac{1}{c} \psi(n) \setend,
\]
then $E_n \subseteq I_n$ and $\m{E_n} = \frac{1}{c} \m{I_n}$.

\end{comment}

With the problem reduced to only considering those sources which are neighborhoods of the origin, it remains to show that the theorem holds for these neighborhoods.

\begin{proposition}
  Let $\psi(n)$ be a function satisfying the conditions in Theorem \ref{thm:p_to_p}, then \[
    \m{\appby{\psi}{B_p^0}} = \begin{cases}
      1 & \sum\limits_{\substack{n=1 \\ p \nmid n }}^\infty n \psi(n) = \infty,\\
      0 & \sum\limits_{\substack{n=1 \\ p \nmid n }}^\infty n \psi(n) < \infty.
  \end{cases}
  \]
\end{proposition}
\begin{proof}
  The convergence case is again a direct Corollary of the classical Khintchine Theorem and we are once again left with the divergence case.
  
  Recall that we are investigating solutions to \[
    \norm{x - \frac{a p^k}{b}} < \psi(b),
  \]
  where $p^{-k}$ is the radius of $B_p^0$, $p \nmid b$, and $a \in \Z$ (i.e. $b$ is $p$-free and $a$ has no other restrictions).
  Rephrasing this, we have a solution to the above if and only if we have a solution to
  \begin{equation}\label{eq:scaled_approximation}
    \norm{ \frac{x}{p^k} - \frac{a}{b} } < \frac{1}{p^k} \psi(b).
  \end{equation}
  Again, since $p$ and $k$ are fixed, by the Cassel's Lemma the measure of the set of $x$ satisfying \eqref{eq:scaled_approximation} is equal to the measure of the set of $x$ satisfying
  \begin{equation}\label{eq:stretched_interval}
    \norm{ \frac{x}{p^k} - \frac{a}{b} } < \psi(b).
  \end{equation}
  Recall the following Corollary of the Duffin--Schaeffer Theorem, which can be found in \cite[pg. 41]{harman_book}; the following is rearranged to match the notation used in the present article.
  
  \begin{corollary}[\cite{harman_book}, Corollary 3]
    Let $\psi(n)$ be a positive-valued function which is non-increasing on a set $\mathcal{A} \subset \N$ with positive lower asymptotic density.
    Then for almost every $x \in \R$ there are infinitely many solutions to \[
      \norm{x - \frac{m}{n}} < \psi(n), \qquad (n,m) = 1, \qquad n \in \mathcal{A},
    \]
    if and only if \[
      \sum_{n \in \mathcal{A}} n \psi(n) = \infty.
    \]
    
    Where the \textbf{lower asymptotic density of a set $\mathcal{A} \subseteq \N$} is defined to be \[
     \liminf_{N \to \infty} \frac{1}{N} \sum_{ \substack{n = 1 \\ n \in \mathcal{A} } }^N 1.
    \]
  \end{corollary}
  In our case we have $\mathcal{A} = \N \setminus p \N$, the set of $p$-free positive integers, which has lower asymptotic density $\frac{p-1}{p}$.
  So, almost every real number, $x$, has infinitely many solutions to \[
    \norm{ x - \frac{a}{b} } < \psi(b),
  \] where $a, b \in \Z \setminus p \Z$.
  Thus, almost every real number, $x$, has infinitely many solutions to \eqref{eq:stretched_interval}.
  Hence, $\m{\appby{\psi}{B_p^0}} = 1$, as needed.
\end{proof}
 
\section{Proof of Theorem \ref{thm:p_to_p}, Case 2}\label{sec:r_to_p}

As in the classical setting, the convergence case will follow by a standard Borel--Cantelli covering argument.
Moreover, sets of the form $\pappby{\psi}{B}$ will satisfy a Zero-One Law (Theorem \ref{thm:zero-one}).
The rest of the argument relies on showing that under the divergence assumption $\pappby{\psi}{B}$ has positive measure.
We do this by showing that a certain subset has positive measure (Proposition \ref{prop:pos_mes}).
For clarity and ease of navigation, the key steps in the argument have been separated into their own sections and subsections.

Here and throughout, we fix a ball $B \subseteq \RmodZ$. 
It will also at times be helpful to use the following notation:
\begin{notation}\label{not:F_n^B}
  Given a set $B$, define the collection \[ \F{n} := \setofall (a,n) \in \N^2 \suchthat (a,n) = 1, p \nmid n, \brackets*{\frac{a}{n}} \in B \setend \]
  to be those pairs where the second coordinate is the fixed number $n$.
  
  Notice that with this notation we can also refer to the $\psi$-approximables by $B$ as \[ \pappby{\psi}{B} := \setofall \alpha \in \Zp \suchthat \pnorm{ n \alpha - a } < \psi(\man{a}{n}), \text{ for i.m. } n \in \N \text{ and } (a,n) \in \F{n} \setend. \]
\end{notation}

\subsection{Convergence}

Suppose $\sum\limits_{n=1}^\infty n \psi(n) < \infty$. 
As in Theorem \ref{thm:p_to_p} Case 1, one may appeal to an established Khintchine's Theorem to determine the convergence case. More directly, one can show this by using a Borel--Cantelli covering argument.
Define the sets \[
  U_n = \bigcup_{a = 1}^n B_p\parens*{ \frac{a}{n}, \psi(n)} \qquad \text{and} \qquad
  V_n = \bigcup_{a = 1}^n B_p\parens*{ \frac{n}{a}, \psi(n)}
\]
and consider their limsup set, $\limsup_{n \to \infty} \parens*{U_n \cup V_n}$. For each $n \in \N$ we have \[
  \m[p]{U_n \cup V_n} \le 2 n \psi(n).
\] Moreover, it's clear that \[
  \pappby{\psi}{B} \subseteq \limsup_{n \to \infty}\, \parens*{ U_n \cup V_n }.
\] By our assumption, \[
  \sum_{n = 1}^\infty 2 n \psi(n) = \sum_{n = 1}^\infty n \psi(n) < \infty
\] and so applying the Borel--Cantelli Lemma we conclude $\m[p]{\pappby{\psi}{B} } = 0$.

\subsection{Zero-One Law}

\begin{theorem}\label{thm:zero-one}
  Suppose $\psi : \Rpos \to \Rpos$ is a monotonically decreasing function such that for all $s > 1$, there exists some $c > 0$ so that for all $n$ sufficiently large $\psi(s n) > c \psi(n)$.
  Then $\m[p]{\pappby{\psi}{B}} = 0 \text{ or } 1$.
\end{theorem}

As in the classical case, it will be helpful to first show that the measure of the $\psi$-approximables is unchanged by scaling the approximating function.

\subsubsection{$\m[p]{\pappby{c \psi}{B}}$ is independent of $c$}

\begin{proposition}\label{prop:ind_of_c}
  Given the assumptions of Theorem \ref{thm:zero-one}, we have \[
    \m[p]{\pappby{c \psi}{B}} = \m[p]{\pappby{\psi}{B}},
  \] for all $c > 0$.
\end{proposition}

\begin{proof}
  We show \[ \m[p]{\pappby{\psi}{B} \setminus \pappby{c \psi}{B}} = 0, \] for all $0 < c < 1$; note that if $c \ge 1$ this follows trivially.\
  We can express $\pappby{c \psi}{B}$ as
  \[ 
    \pappby{c \psi}{B} = \bigcap_{N = 1}^\infty W_p(c,N)
  \]
  where 
  \begin{align*}
    W_p(c,N) 
      &:= \setofall \alpha \in \Zp \suchthat \exists n \in \N, \exists (a,n) \in \F{n}, \man{a}{n}\ge N, \pnorm{ n \alpha - a } \le c \cdot \psi( \man{a}{n}) \setend \\
      &= \bigcup_{ \substack{n \in \N \\ (a,n) \in \F{n} \\ \man{a}{n}\ge N } }  \setofall \alpha \in \Zp \suchthat \pnorm{ n \alpha - a } \le c \psi( \man{a}{n}) \setend.
  \end{align*}
  By definition, $W_p(c,N)$ and $\pappby{c \psi}{B}$ are measurable. Since $\Zp$ is a finite measure space, by the Downward Monotone Convergence Theorem, for any $\eps > 0$ there exists an $N(\eps)$ such that for all $N \ge N(\eps)$ we have \[ \m[p]{ W_p(c,N) \setminus \pappby{c \psi}{B} } \le \eps. \]

  Now, towards a contradiction, suppose that $\m[p]{ \pappby{\psi}{B}- \pappby{c \psi}{B} } > 0$. Then we may find an $N$ so that $\m[p]{ \pappby{\psi}{B}- W_p(c,N)} > 0$.
  Moreover, by a $p$-adic analogue to the Lebesgue Density Theorem (see \cite{lebesgue-density} for details), almost every element of the set $\pappby{\psi}{B} - W_p(c, N)$ is a density point. Let $\alpha' \in \pappby{\psi}{B} - W_p(c, N)$ be such a point.
  It then follows that
  \[
    \frac{ \m[p]{ B_p(\alpha', p^{-l} ) \cap W_p(c,N) } }{ \m[p]{B_p(\alpha', p^{-l} )} } \to 0 \qquad \text{as} \qquad l \to \infty,
  \]
  or, equivalently, $p^{-l} \to 0$.
  So, there exists $l_0 \in \N$ such that for all $l \ge l_0$
  \begin{equation}\label{eq:pos_mes_contradiction}
    \m[p]{ B_p(\alpha', p^{-l} ) \cap W_p(c,N) } < p^{- l - l'}
  \end{equation}
  where $l' \in \N$ is such that $p^{-l'} \le c < p^{-l' + 1} \le 1$.
  Now, since we have such an element $\alpha'$ there must exist $(a',n') \in \F{n'}$ with $\psi(\man{a'}{n'}) \le p^{l_0}$ and an $l_1 > l_0$ so that \[ 
    \pnorm{ n' \alpha' - a' } \le p^{-l_1} \le \psi( \man{a'}{n'} ),
  \]
  with $\man{a'}{n'} > N$.

  Consider the set \[ U_{(a',n')}(l_1) := \setofall \alpha \in \Zp \suchthat \pnorm{n' \alpha - a'} \le c p^{-l_1} \le c \psi(\man{a'}{n'}) \setend.
  \]
  Notice that this set can also be expressed as the ball about $\frac{a'}{n'}$ of radius $p^{-l'} p^{-l_1} = p^{-l_1 - l'} < c p^{-l_1}$ (because $p^{-l'} \le c < p^{-l' + 1}$).
  Also, note that since we have $U_{(a',n')}(l_1) \subseteq W_p(c,N)$ we also have \[ U_{(a',n')}(l_1) \cap B_p(\alpha', p^{-l_1}) \subseteq W_p(c,N) \cap B_p(\alpha', p^{-l_1}). \]
  Since both sets on the left-hand side are $p$-adic balls and they intersect (because $\pnorm{n' \alpha' - a'} \le p^{-l_1}$), the intersection is exactly the smaller ball; This is a particular property for balls in $\Qp$. Specifically, because $p^{-l_1 - l'} \le p^{-l_1}$ we have \[
    \m[p]{ U_{(a',n')}(l_1) \cap B_p(\alpha', p^{-l_1}) } = p^{-l_1 - l'}
  \]
  This immediately contradicts \eqref{eq:pos_mes_contradiction} since it implies \[
    p^{-l_1 - l'} 
      > \m[p]{ W_p(c,N) \cap B_p(\alpha', p^{-l_1}) }
      \ge \m[p]{ U_{(a',n')}(l_1) \cap B_p(\alpha', p^{-l_1}) } 
      = p^{-l_1 - l'}.
  \]

  Thus, $\m[p]{ \pappby{\psi}{B} \setminus \pappby{c \psi}{B}} = 0$; that is, the measure is independent of the scaling constant $c$.
\end{proof}

\subsubsection{$\m[p]{\pappby{\psi}{B}}$ is either $0$ or $1$}

With the above in our toolbelt, we can now prove Theorem \ref{thm:zero-one}.
\begin{proof}
If $\m[p]{\pappby{c \psi}{B}} = 0$, we are done.
Suppose instead that $\m[p]{\pappby{c \psi}{B}} > 0$. 
Given any $\eps > 0$, define the auxiliary set \[ W'(\eps) := \setofall \alpha = \alpha' + z \in \Zp \suchthat \alpha' \in \pappby{\psi}{B}, z \in \Z, \abs{z} \le c_1 \setend \] where $c_1 = c_1(\eps) > 0$ is chosen using Lemma \ref{lem:lutz1.1} to guarantee \[ \m[p]{ W'(\eps) } \ge 1 - \eps. \]

To conclude the argument, we show that $W'(\eps) \subseteq \pappby{d \psi}{B}$, for some $d$. Then, by our previous work, Proposition \ref{prop:ind_of_c}, we have \[
  \m[p]{\pappby{\psi}{B}} = \m[p]{\pappby{d \psi}{B}} \ge 1 - \eps,
\] for all $\eps > 0$, and thus $\m[p]{\pappby{\psi}{B}} = 1$, as needed.

Let $\alpha = \alpha' + z \in W'(\eps)$ be given; that is, $\alpha' \in \pappby{\psi}{B}$ and $\abs{z} \le c_1$.
By rearranging these quantities we have \[ n \alpha' - a = n (\alpha - z) - a = n \alpha -  (a + z n). \]
Then, by defining $a' = a + z n$ and $n' = n$ we have
\begin{equation}\label{eq:new_same_norm}
  \pnorm{ n \alpha' - a } = \pnorm{ n' \alpha - a' }.
\end{equation}
Moreover, we have
\begin{align*}
  \man{a'}{n'}
    &= \man{ a + zn }{ n } \\
    &\le \man{ \abs{a} + \abs{zn} }{ n } \\
    &\le \man{a}{n} + c_1 \man{a}{n} \\
    &= (1 + c_1) \man{a}{n}= k \man{a}{n}
\end{align*}
by taking $k = 1 + c_1$; that is, $\frac{\man{a'}{n'}}{k} \le \man{a}{n}$.

Now, since $\alpha' \in \pappby{\psi}{B}$ there exist infinitely many $n \in \N$ with a solution $(a,n) \in \F{n}$ to 
\begin{equation}\label{eq:new_upper_bound}
  \pnorm{ n \alpha' - a } \le \psi( \man{a}{n}).
\end{equation}
However, by our definition of $(a', n')$ and our assumption that $\psi(n)$ is decreasing we have
\begin{align*}
  \pnorm{ n' \alpha - a' } 
    \stackrel{\eqref{eq:new_same_norm}}{=} 
    \pnorm{ n \alpha' - a } 
    \stackrel{\eqref{eq:new_upper_bound}}{\le} \psi\parens*{ \man{a}{n} } 
    \le \psi\parens*{ \frac{\man{a'}{n'}}{k} }.
\end{align*}
Moreover, since we have a bijection between the pairs $(a,n)$ and $(a', n')$ there are infinitely many solutions to the above inequality. 
Note that 
\begin{equation}\label{eq:integer_translate}
  \frac{a'}{n'} = \frac{a + zn}{n} = \frac{a}{n} + z,
\end{equation}
is an integer translate of $\frac{a}{n}$, so \[
  \brackets*{ \frac{a'}{n'} } = \brackets*{ \frac{a}{n} } \in B.
\]

Recall our regularity condition on $\psi(n)$, by a change of variables it is equivalent to: for all $s < 1$, there exists $c_1 > 0$ so that for all sufficiently large $n$ we have $\psi(sn) < c_1 \psi(n)$.
With this reformulation, we conclude that there exists some $d = d(k) > 0$ so that we have infinitely many solutions $(a', n') \in \F{n'}$ to \[
  \pnorm{ n' \alpha - a'} \le \psi\parens*{ \frac{\man{a'}{n'}}{k} } \le d \cdot \psi( \man{a'}{n'} ).
\]
That is $\alpha \in \pappby{d \psi}{B}$, and so $W'(\eps) \subseteq \pappby{d \psi}{B}$.
Thus, \[
  \m[p]{\pappby{\psi}{B}} = \m[p]{\pappby{d \psi}{B}} \ge \m[p]{W'(\eps)} \ge 1 - \eps,
\] for all $\eps > 0$.
\end{proof}

This concludes the proof of Theorem \ref{thm:zero-one}; that is, the measure of the $\psi$-approximables can only be zero or one.

\subsection{$\pappby{\psi}{B}$ has positive measure}\label{sec:pos_mes}

Instead of showing $\pappby{\psi}{B}$ has positive measure, we show that a conveniently chosen proper subset has positive measure.

\begin{proposition}\label{prop:pos_mes}
  If $\sum\limits_{\substack{n=1\\p\nmid n}}^\infty n \psi(n) = \infty$, then the set $\pappby{\psi}{B}$ has positive measure.
\end{proposition}
\begin{proof}
  Following the style of Notation \ref{not:F_n^B}, we denote by $B^{[0,1]}$ the collection of coset representatives of $B$ that lie in the interval $[0,1]$.
  Similarly, we denote by $\F[B^{[0,1]}]{n}$ the elements of $\F[B]{n}$ where $\frac{a}{n}$ is restricted to $[0,1]$.
  With this notation set, it suffices to show \[
    \m[p]{\pappby{\psi}{B^{[0,1]}}} > 0
  \] since
  \[ \pappby{\psi}{B^{[0,1]}} \subseteq \pappby{\psi}{B}. \] 
  
  For our  convenience, we will write $\Psi(N) = \sum\limits_{n = 1}^N n \psi(n)$ and define $\Delta_N := \Delta_N(\alpha)$ to be the number of solutions to $\pnorm{n \alpha - a} < \psi(\man{a}{n})$ where $n \le N$ and $(a,n) \in \F[B^{[0,1]}]{n}$. 
  Also, we define \[ \dan(\alpha) := \chi_{\set{ \alpha' \in \Zp \suchthat \pnorm{n \alpha' - a} < \psi(\man{a}{n}) } } (\alpha) , \] as the characteristic function for the ball $B_p\parens*{\frac{a}{n}, \psi(\man{a}{n})}$.
  Then, we can rewrite the number of solutions as \[ \Delta_N = \sum_{\substack{n = 1 \\ (a,n) \in \F[B^{[0,1]}]{n}}}^N \dan. \]
  Lastly, we define the following two quantities:
  \[ M_1 = M_1(N) := \int_{\Zp} \Delta_N(\alpha) \da, \]
  and
  \[ M_2^2 = M_2^2(N) := \int_{\Zp} \Delta_N^2(\alpha) \da. \]

  With this setup, the proof is now reduced, in essence, to satisfying the conditions of the Paley--Zygmund Lemma. 
  The statement and full proof of the Lemma can be found on page 122 in \cite{cassels}.
  \begin{lemma}[Paley--Zygmund Lemma]\label{lem:paley_zygmund}
    Suppose $f : \Qp \to \R$ is a non-negative, integrable function. Define
    \[
    M_1 := \int_{\Zp} f(\alpha) \da \qquad \text{and} \qquad M_2^2 := \int_{\Zp} f^2(\alpha) \da.
    \]
    If there exists $c_1 > 0$ such that $c_1 M_2 \le M_1$ then for all $0 \le c_2 \le c_1$, the set of points for which $f(\alpha) \ge c_2 M_2$ has measure at least $(c_1 - c_2)^2$; that is, \[ \m[p]{ \setofall \alpha \in \Zp \suchthat f(\alpha) \ge c_2 M_2 \setend} \ge (c_1 - c_2)^2. \]
  \end{lemma}
  
  In particular, we show that there is a constant $N_0 = N_0(B)$, depending solely on $B$, so that the Paley--Zygmund Lemma may be applied to all $M_1(N)$, $M_2(N)$ such that $N \ge N_0(B)$; This constant is defined in Lemma \ref{lem:N0}.
  With that, we have that there exists a $c_1$ (depending only on $B$) such that $c_1 M_2(N) \le M_1(N)$, for all $N \ge N_0(B)$. To conclude that $\pappby{\psi}{B}$ has positive measure, take any $c_2 \in (0,c_1)$, by the Paley--Zygmund Lemma, there is a positive measure set in $\Zp$ so that every $\alpha$ in that set satisfies \[ \Delta_N(\alpha) \ge c_2 M_2(N),\] for all $N \ge N_0(B)$.
  Moreover, by Cauchy's Inequality, we have that $M_2(N) \ge M_1(N)$ for all $N \ge N_0(B)$.
  Lastly, we show in the next section that $M_1(N) \ge \Psi(N)$, again, for all $N \ge N_0(B)$.
  All together, we have the chain of inequalities \[
    \Delta_N(\alpha) \ge c_2 M_2(N) \ge c_2 M_1(N) \ge c_2\Psi(N),
  \]
  for a positive measure set of $\alpha \in \Zp$.
  Since we are assuming that $\Psi(N) = \sum\limits_{n=1}^\infty \psi(n)$ diverges, we have $\Delta_N(\alpha) \to \infty$ as $N \to \infty$, on this positive measure set of $\alpha \in \Zp$. Thus, each $\alpha$ in this positive measure set has an infinite number of solutions of the desired form. 
  
  We now produce the respective lower and upper bounds.
  
  \subsubsection{Bounding $M_1$}
  
  Consider the following chain of equalities
  \begin{align*}
    M_1
      &= \int_{\Zp} \Delta_N \da \\
      &= \int_{\Zp} \sum_{\substack{n = 1 \\ (a,n) \in \F[B^{[0,1]}]{n}}}^N \dan \da \\
      &= \sum_{\substack{n = 1 \\ (a,n) \in \F[B^{[0,1]}]{n}}}^N  \int_{\Zp} \dan \da \\
      &= \psum_{n = 1}^N \sum_{ (a,n) \in \F[B^{[0,1]}]{n}} \int_{\Zp} \dan \da,
      \intertext{where $\psum$ is the sum restricted to values not divisible by $p$,}
      &= \psum_{n = 1}^N \sum_{ (a,n) \in \F[B^{[0,1]}]{n}} \psi^*(n)
      \intertext{where $t^*$ is the unique power of $p$ such that $\frac{t}{p} < t^* \le t$,}
      &= \psum_{n = 1}^N \bphi(n) \psi^*(n), \\
  \end{align*}
  and $\bphi[B](n) = \#\setofall a \in \N \suchthat (a,n) = 1, \frac{a}{n} \in [0,1] \cap B \setend$. 
  That is, $\bphi(n)$ is a modified version of the standard Euler totient function; specifically $\bphi[{[0,1]}](n)$ is \emph{exactly} the standard Euler totient function, $\varphi(n)$.
  With this notation, the last equality follows from the observation that by definition we have $\abs{ \F[B^{[0,1]}]{n} } = \bphi(n)$, when $p \nmid n$; in sum, the above calculation yields
  \begin{equation}\label{eq:M_1_equality}
    M_1 = \psum_{n = 1}^N \bphi(n) \psi^*(n).
  \end{equation}
  
  Following a similar result for the classical Euler totient function, we have the generalization
  \begin{lemma}\label{lem:N0}
    There exists $N_0(B) \in \N$ and a constant $C = C(B)$, such that for all $N \ge N_0(B)$
    \[ \psum_{n=1}^N \bphi(n) \ge C N^2 \]
    for all $N \ge N_0(B)$.
  \end{lemma}
  \begin{proof}
    By equidistribution of the Farey sequence, see \cite{niederreiter}, we have $\frac{ \bphi(n) }{ \varphi(n) } \to \lambda(B)$ as $n \to \infty$; where $\lambda(-)$ is the Lebesgue measure on $\R$.
    We can choose $N_0(B)$ sufficiently large so that \[
      \bphi(n) \ge \frac{1}{2} \lambda(B) \varphi(n),
    \] for all $n \ge N_0(B)$.
    Given $N_0(B)$, there is some $C_1$ so that \[
      \psum_{n=1}^{N_0 - 1} \bphi(n) \ge C_1 (N_0 - 1)^2.
    \]
    Note that \[
      \psum_{n=1}^{N_0 - 1} \varphi(n) \le \sum_{n=1}^{N_0 - 1} \varphi(n) \le (N_0 - 1)^2.
    \] 
    By the above we have
    \begin{align*}
      \psum_{n=1}^N \bphi(n)
        &= \psum_{n=1}^{N_0 - 1} \bphi(n) + \psum_{N_0}^N \bphi(n) \\
        &\ge C_1 (N_0 - 1)^2 + \frac{1}{2} \lambda(B) \psum_{N_0}^N \varphi(n) \\
        &\ge C_1 \psum_{n=1}^{N_0 - 1} \varphi(n) + \frac{1}{2} \lambda(B) \psum_{N_0}^N \varphi(n) \\
        &\ge C_2 \psum_{n=1}^N \varphi(n)
    \end{align*}
    where $C_2 = \min\set{C_1, \frac{1}{2} \lambda(B)}$.
    Applying Lutz's Lemma 4.16 in \cite{lutz} there is some $C_3 > 0$ so that for all $N$ we have \[
      \psum_{n=1}^N \varphi(n) \ge C_3 N^2.
    \]
    
    Taking $C = C_2 \cdot C_3$ and any $N \ge N_0$, we have \[
      \psum_{n=1}^N \bphi(n) \ge C N^2,
    \] as needed.
  \end{proof}
  
  We now recall a result of Lutz's which relates $p$-free sums to their full counterparts.  
  \begin{corollary}[{\cite[Corollary to Lutz's Lemma 4.17]{lutz}}]\label{cor:lutz4.17}
    Suppose $f, g : \Zpos \to \Rpos$ are such that $f$ is decreasing and there exists a constant $C > 0$ and an integer $\rho > 0$ such that for all $N \ge N_0(B)$,
    \[
    \psum_{n = 1}^N g(n) > C N^{\rho + 1},
    \]
    again $\psum$ is still the sum restricted to values not divisble by $p$.
    Then, for all $N \ge N_0(B)$ we have
    \[ 
      \psum_{n=1}^N f(n) g(n) > C \sum_{n=1}^N n^\rho f(n).
    \]
  \end{corollary}
  
  \begin{remark}
    Note there is a minor difference between our statement of the Corollary and the original statement in \cite{lutz}. In the original, one requires the inequality to hold for all $N \ge 1$ but one may just as well loosen that restriction to merely holding for all $N \ge N_0(B)$.
    The proof of this generalization follows immediately on repeating the argument in \cite{lutz}.
  \end{remark}
  
  From these two lemmas, taking $f = \psi(n)$, $g(n) = \bphi(n)$, and $\rho = 1$, we get a constant $C > 0$ so that for all $N \ge N_0(B)$, we have \[ \psum_{n=1}^N \bphi(n) \psi(n) \ge C \sum_{n=1}^N n \psi(n) = C  \Psi(N). \]
  Combining this with \eqref{eq:M_1_equality} we get
  \begin{equation}\label{eq:M_1_bound}
     M_1 
      = \psum_{n = 1}^N \bphi(n) \psi^*(n) 
      \ge \frac{1}{p} C \sum_{n = 1}^N n \psi(n) 
      \ge \parens*{ \frac{C}{p} } \Psi(N).
  \end{equation}
  
  \subsubsection{Bounding $M_2^2$}
  
  To bound the quantity $M_2^2$ we first decompose it into more manageable components:
  \begin{align*}
    M_2^2
      &= \int_{\Zp} \Delta_N^2 \da \\
      &= \int_{\Zp} \parens*{ \sum_{\substack{n = 1 \\ (a,n) \in \F[B^{[0,1]}]{n}}}^N \dan }^2 \da \\
      \intertext{since $\dan$ is a characteristic function we have}
      &= \int_{\Zp} \sum_{n,m = 1}^N \sum_{ \substack{ (a,n) \in \F[B^{[0,1]}]{n} \\ (b,m) \in \F[B^{[0,1]}]{m} } } \dan \cdot \dan[b][m] \da \\
      &= \sum_{n,m = 1}^N \sum_{ \substack{ (a,n) \in \F[B^{[0,1]}]{n} \\ (b,m) \in \F[B^{[0,1]}]{m} } } \int_{\Zp} \dan \dan[b][m] \da \\
      &= \underbrace{ \sum_{n=1}^N \sum_{ (a,n) \in \F[B^{[0,1]}]{n} } \int_{\Zp} \dan \da }_{(a,n) = (b,m)}
        + \underbrace{ \sum_{n,m = 1}^N \sum_{ \substack{ (a,n) \in \F[B^{[0,1]}]{n} \\ (b,m) \in \F[B^{[0,1]}]{m} \\ (a,n) \ne (b,m)} } \int_{\Zp} \dan \dan[b][m] \da}_{(a,n) \ne (b,m)}.
  \end{align*}
  In the case of equality, we have as a natural upperbound
  \begin{equation}\label{eq:a,n=b,m}
    \sum_{n=1}^N \sum_{ (a,n) \in \F[B^{[0,1]}]{n} } \int_{\Zp} \dan \da \le \sum_{n=1}^N n \psi(n) = \Psi(N).
  \end{equation}
  
  Recall that in $\Qp$ exactly one of the following occurs for any pair of balls: the balls are disjoint or one of the balls is contained in the other.
  Thus, we have
  \begin{equation}\label{eq:meas padic ball is min}
    \int_{\Zp} \dan \dan[b][m] \da 
      = \begin{cases}
        \min\set{\psi^*(n), \psi^*(m)} & \pnorm{\frac{a}{n} - \frac{b}{m} } < \max\set{\psi^*(n), \psi^*(m) } \\
        0 & \pnorm{\frac{a}{n} - \frac{b}{m} } \ge \max\set{\psi^*(n), \psi^*(m) }.
    \end{cases}
  \end{equation}
  That is, the measure of the ball is either $0$ or the measure of the smaller ball depending on whether the balls intersect or not. 
  One must then ask how many such intersections exist; that is, given the conditions on $a,b,n,m$ above, how often is $\pnorm{\frac{a}{n} - \frac{b}{m}} < \max\set{\psi^*(n), \psi^*(m)}$? The following two lemmas give us upper bounds.
  \begin{lemma}{\cite[Lemma 4.14]{lutz}}\label{lem:lutz4.14}
    If $m \ne n$ and $p \nmid m,n$ then the number of $(a,b) \in \N^2$ such that \[ 
      \pnorm{\frac{a}{n} - \frac{b}{m}} < \max\set{\psi^*(n), \psi^*(m)}
    \]
    is bounded above by $4nm \max\set{\psi^*(n), \psi^*(m)}$.
  \end{lemma}
  \begin{lemma}\label{lem:lutz4.14_missing}
    If $m = n$, $a \ne b$ and $p \nmid m,n$ then the number of $(a,b) \in \N^2$ such that \[ 
      \pnorm{\frac{a}{n} - \frac{b}{m}} < \max\set{\psi^*(n), \psi^*(m)}
    \]
    is bounded above by $4nm \max\set{\psi^*(n), \psi^*(m)}$.
  \end{lemma}
  \begin{proof}
    Since $p \nmid m,n$ we have 
    \begin{align}
      \pnorm{ \frac{a}{n} - \frac{b}{m} }
        &\le \pnorm{ \frac{ am - bn }{mn} }
        = \pnorm{ am - bn },
    \end{align}
    and since $m = n$ is an integer we have \[
      \pnorm{ am - bn } = \pnorm{ am - bm } = \pnorm{ m (a - b) } = \pnorm{ a - b }.
    \]
    We show that the number of solutions $(a,b)$ to 
    \begin{equation}\label{eq:lutz4.14_inequality}
      \pnorm{a - b} < \psi^*(n)
    \end{equation}
    is bounded above by $n^2 \psi^*(n)$ and since $n=m$ this completes the proof.
    
    Recall that $\psi^*(n)$ is the unique power of $p$ so that $\frac{\psi(n)}{p} < \psi^*(n) \le \psi(n)$, say $\psi^*(n) = p^{-t}$.
    Since $a$ and $b$ are integers the inequality in \eqref{eq:lutz4.14_inequality} can be rephrased as \[
      a - b = p^t c
    \] for some integer $c$.
    There are $n$ choices for $a$ and thus for a fixed value of $a$ there at most $\frac{n}{p^t}$ choices for $b$ satisfying this condition.
    That is, there no more than \[
      n \cdot \frac{n}{p^{-t}} = n^2 \psi^*(n)
    \] pairs $(a,b)$ satisfying \eqref{eq:lutz4.14_inequality}, as needed.
  \end{proof}
  
  On combining these with \eqref{eq:meas padic ball is min} and our bound in \eqref{eq:a,n=b,m} we have
  \begin{align*}
    M_2^2 
      &= \sum_{n=1}^N \sum_{ (a,n) \in \F[B^{[0,1]}]{n} } \int_{\Zp} \dan \da
        + \sum_{n,m = 1}^N \sum_{ \substack{ (a,n) \in \F[B^{[0,1]}]{n} \\ (b,m) \in \F[B^{[0,1]}]{m} \\ (a,n) \ne (b,m)} } \int_{\Zp} \dan \dan[b][m] \da \\
      &\le \Psi(N) + \sum_{n,m=1}^N 4nm \max\set{\psi^*(n), \psi^*(m)} \cdot \min\set{\psi^*(n), \psi^*(m)} \\
      &= \Psi(N) + 4 \sum_{n,m=1}^N nm \psi^*(n) \psi^*(m) \\
      &\le \Psi(N) + 4\sum_{n,m=1}^N nm \psi(n) \psi(m) \\
      &\le \Psi(N) + 4 \Psi(N)^2 \\
      &= \parens*{ \frac{1}{\Psi(N)} + 4 } \Psi(N)^2
  \end{align*}
  More succinctly, we have \[ M_2 \le 2 \Psi(N). \]
  Lastly, observe that this implies \[ M_1 \ge \frac{C}{p} \Psi(N) = \frac{C}{2p} \cdot 2 \Psi(N) \ge \frac{C}{2p} \cdot M_2, \] which concludes the proof by taking $c_1 = \frac{C}{2p}$ in the Paley--Zygmund Lemma.
\end{proof}

\section{Proof of Theorem \ref{thm:p_to_p}, Case 3}\label{sec:p_to_p}

The proof of Case 3 follows the same argument as Case 2, with some minor modifications.
Because of the similarities, we have extracted the changes and include them here, with their proofs, to maintain the flow of the argument in Case 2 and allow the reader to substitute in the required pieces as needed.

As in the previous cases, the convergence case follows from a standard Borel--Cantelli covering argument.
What remains are modifications to Notation \ref{not:F_n^B}, the (already) modified Euler totient function, Lemma \ref{lem:lutz1.1}, the proof of Theorem \ref{thm:zero-one}, and Lemma \ref{lem:N0}.

We begin with the modification to Notation \ref{not:F_n^B}:
\begin{notation}[Modified Notation \ref{not:F_n^B}]\label{not:F_n^b_modified}
  Given a set $B$, define the collection \[ \F[B_{p_1}]{n} := \setofall (a,n) \in \N^2 \suchthat (a,n) = 1, \frac{a}{n} \in B_{p_1} \setend \]
  to be those pairs where the second coordinate is the fixed number $n$.
  
  With this notation we can refer to the $\psi$-approximables by $B_{p_1}$ as \[ \pappby[p_2]{\psi}{B_{p_1}} := \setofall \alpha \in \Zp[p_2] \suchthat \pnorm[p_2]{ n \alpha - a } < \psi(\man{a}{n}), \text{ for i.m. } n \in \N \text{ and } (a,n) \in \F[B_{p_1}]{n} \setend. \]
\end{notation}

With this we re-modify the Euler totient function as \[
  \bphi[B_{p_1}] = \# \setofall a \in \N \suchthat (a,n) = 1, \frac{a}{n} \in [0,1] \cap B_{p_1} \setend.
\]
It then follows that \[
  \abs{ \F[B_{p_1}^{[0,1]}]{n} } = \bphi[B_{p_1}],
\]
where, as before, $B_{p_1}^{[0,1]} = B_{p_1} \cap [0,1]$.

We also require the following modification of Lutz's Lemma 1.1 to account for the change in permissible translations.

\begin{lemma}[Modified Lemma \ref{lem:lutz1.1}]\label{lem:lutz1.1_modified}
  Let $p_1, p_2$ be distinct primes and $k \in N$ be any positive integer. 
  Suppose $A \subseteq \Zp[p_2]$ is a positive measure set. Given any $\eps > 0$, there exists $c = c(\eps) > 0$ such that the set \[
    \setofall \alpha = \alpha' + z \suchthat \alpha' \in A, z \in p_1^k \Z, \abs{z} < c
    \setend
  \]
  has measure at least $1 - \eps$.
\end{lemma}

\begin{proof}
  Let $A \subseteq \Zp[p_2]$ be a positive measure set and $\eps > 0$ be given.
  By a $p_2$-adic analogue to the Lebesgue Density Theorem (again, see \cite{lebesgue-density} for details) almost every element of $A$ is a density point.
  Let $\alpha \in A$ be such an element.
  It then follows that \[
    \frac{\m[p_2]{ B(\alpha, p_2^{-l}) \cap A^c}}{\m[p_2]{ B(\alpha, p_2^{-l}) } } \to 0
  \]
  as $l \to \infty$ (i.e. as the radius $p_2^{-l}$ tends to zero) where $A^c$ is the complement of the set $A$.
  Letting $U_\alpha = U_{\alpha, l} = B(\alpha, p_2^{-l}) \cap A^c$, we get that there exists some integer $l_0 = l_0(\eps)$ such that \[
    \frac{\m[p_2]{ U_\alpha } }{\m[p_2]{ B(\alpha, p_2^{-l}) } } \le \eps
  \]
  for $l \ge l_0$ and so
  \begin{equation}\label{eq:density_point_upperbound}
    \m[p_2]{ U_\alpha } \le \eps \m[p_2]{ B(\alpha, p_2^{-l}) }.
  \end{equation}
  Fix such an $l \ge l_0$.
  Recall that in $\Zp[p_2]$, there are exactly $p_2^l$ disjoint balls of radius $p_2^{-l}$ which partition $\Zp[p_2]$.
  So, the collection of translates $\set{B(\alpha + z, p_2^{-l})}$, for $z = 0, 1, 2, \cdots, p_2^l - 1$, form a partition of $\Zp[p_2]$.
  Notice that the collection of translates $\set{B(\alpha + z, p_1^k p_2^{-l})}$, for $z = 0, 1 p_1^k, 2 p_1^k , \cdots, (p_2^l - 1) p_1^k$ also form a partition of $\Zp[p_2]$.
  
  To see this, take any 2 centers of the original balls $0 \le i \le j < c = p_2^l$. The distance between the new centers $i p_1^k$ and $j p_1^k$ has not changed: \[
    \pnorm[p_2]{i p_1^k - j p_1^k}
      = \pnorm[p_2]{ (i - j) \cdot p_1^k}
      = \pnorm[p_2]{i - j}.
  \]
  Since this is true for every pair of centers, and (in $\Zp[p_2]$) any element in a ball is also the center of that ball, all of the balls must still be disjoint.
  
  Thus, the set of elements in $\Zp[p_2]$ that are not contained in the set of translates \[
    \bigcup_{\substack{z = 0 \\ p_1^k \mid z}}^{p_1^k p_2^l} A + z
  \]
  is exactly the set
  \[
    \bigcup_{\substack{z = 0 \\ p_1^k \mid z}}^{p_1^k p_2^l} U_{\alpha + z}.
  \]
  
  However, the measure of this set (i.e. the set of points not contained in the translates of $A$) is
  \begin{align*}
    \m[p_2]{\bigcup_{\substack{z = 0 \\ p_1^k \mid z}}^{p_1^k \cdot p_2^l } U_{\alpha + z}} \\
      &= \m[p_2]{\bigcup_{z = 0}^{p_2^l} U_{\alpha + p_1^k z}} \\
      &\le \sum_{0 \le z < p_2^l} \m[p_2]{U_{\alpha + p_1^k z}} \\
      &= \sum_{0 \le z < p^l} \m[p_2]{U_{\alpha}}, \\
      \intertext{since translations are measure-preserving,}
      &= p_2^l \cdot \m[p_2]{U_{\alpha}} \\
      &\le p_2^l \cdot \eps \m[p_2]{ B(\alpha, p_2^{-l}) } &\text{by \eqref{eq:density_point_upperbound}} \\
      &= p_2^l \cdot \eps \cdot p_2^{-l} \\
      &= \eps.
  \end{align*}
  Thus, by taking $c = p_1^k \cdot p_2^l$, we conclude that \[
    \m[p_2]{ \bigcup_{0 \le z < c} A + z } \ge 1 - \eps,
  \]
  as needed.
\end{proof}

Given this modified Lemma and the new notation outlined above, the proof of Theorem \ref{thm:zero-one} follows exactly as before with the observation that the translated rationals in \eqref{eq:integer_translate} fall inside $B_{p_1}$ as before, since we are translating by an element of $p_1^k \Z$.

Lastly, we include the modification of Lemma \ref{lem:N0}.

\begin{lemma}[Modified Lemma \ref{lem:N0}]
  There exists $N_0(B_{p_1}) \in \N$ and a constant $C = C(B_{p_1})$, such that for all $N \ge N_0(B_{p_1})$
  \[ \psum_{n=1}^N \bphi[B_{p_1}](n) \ge C N^2 \]
  for all $N \ge N_0(B_{p_1})$.
\end{lemma}
\begin{proof}
  Since the integers are dense in $\Zp[p_1]$, the rationals in the ball $B_{p_1}$ can be expressed as $r + \frac{a}{n}p_1^l$, where $r \in \Z$ is the center of the ball and $p_1^{-l}$ is its radius.
  We will show that \[
    \bphi[B_{p_1}](n) \ge \bphi[B](n),
  \] for all $n$ not divisible by $p$, where $B = B\parens*{0, \frac{r}{p_1^l}}$, and then apply Lemma \ref{lem:N0} to the ball, $B$.
  
  By definition we have
  \begin{align*}
    \bphi[B_{p_1}](n)
      &= \# \setofall a \in \N \suchthat (a,n) = 1, \frac{a}{n} \in [0,1] \cap B_{p_1} \setend \\
      &\ge \# \setofall 0 \le a \le n \suchthat (a,n) = 1, \frac{a}{n} \in B_{p_1} \setend. \\
      \intertext{Since the rationals in $B_{p_1}$ are of the form $r + \frac{a}{n}p_1^l = \frac{rn + a p_1^l}{n}$, we have,}
        \bphi[B_{p_1}](n) &\ge \# \setofall 0 \le r n + a p_1^l \le n \suchthat (rn + a p_1^l,n) = 1 \setend. \\
      \intertext{Since $p_1 \nmid n$ it follows that if $(a,n) = 1$, then $(rn + ap_1^l, n) = 1$, so}
        \bphi[B_{p_1}](n) &\ge \# \setofall 0 \le r n + a p_1^l \le n \suchthat (a,n) = 1 \setend.
  \end{align*}
  Then, since $0 \le rn + a p_1^l$, we must have $a \le \frac{rn}{p_1^l}$, from which we get the lower bound \[
    \# \setofall 0 \le r n + a p_1^l \le n \suchthat (a,n) = 1 \setend
      \ge \# \setofall 0 \le a \le \frac{rn}{p_1^l} \suchthat (a,n) = 1 \setend.
  \]
  Now, observe that this last set is exactly counting the reduced rationals in the interval $\brackets*{0, \frac{r}{p_1^l}}$.
  Thus, we have
  \begin{align*}
    \bphi[B_{p_1}](n)
      &\ge \# \setofall 0 \le a \le \frac{rn}{p_1^l} \suchthat (a,n) = 1 \setend \\
      &\ge \# \setofall a \in \N \suchthat (a,n) = 1, \frac{a}{n} \in [0,1] \cap B \setend
      = \bphi[B](n),
  \end{align*}
  as needed.
\end{proof}

\section{Hausdorff Measure Generalizations}\label{sec:hausdorff_generalization}

For completeness, we define the preliminaries needed in stating the Mass Transference Principle, Theorem \ref{thm:mtp}. We do not need the full force of the Mass Transference Principle. In our application, we only consider the the case $g(x) = x$, so $\Haus[g] = \Haus[1]$ using the notation in \cite{mtp}. We will see shortly that $\Haus[1]$ is exactly the standard Haar measure on $\Zp$ as well as the standard Lebesgue measure on $\R$ (when defined for the respective spaces).

\subsection{Mass Transference Principle}

Throughout, $(X, d)$ is a locally compact metric space such that for every $\rho > 0$ the space can be covered by a countable collection of balls with diameters less than $\rho$.

A \emph{dimension function} $f : \Rpos \to \Rpos$ is a continuous, nondecreasing function such that $f(r) \to 0$ as $r \to 0$.
A function is \emph{doubling} if there exists a constant $\lambda > 1$ such that for $x > 0$ \[
  f(2x) \le \lambda f(x).
\]
Given a doubling dimension function $g$, a dimension function $f$, and a ball $B = B(x,r)$ we define \[
  B^f := B\parens*{x, g^{-1}\parens*{f(r)} }.
\]
Note that for the case $g(x) = x$, we have $B^f = B(x, f(r))$.
Given a ball $B = B(x,r)$ in $X$ and a dimension $f$, we define the \emph{$f$-volume of $B$} to be \[
  V^f(B) := f(r).
\]

Suppose $F \subseteq X$, let $f$ be a dimension function and let $\rho > 0$.
A \emph{$\rho$-cover for $F$} is any countable collection of balls $\set{B_i}_{i \in \N}$ with $\text{radius}(B_i) < \rho$ for every $i \in \N$ and $F \subseteq \bigcup_{i \in \N} B_i$.
We define \[
  \Haus[f]_\rho(F) := \inf \setofall \sum_{i} V^f(B_i) \suchthat \set{B_i} \text{ is a $\rho$-cover for $F$} \setend.
\]
The \emph{Hausdorff $f$-measure of $F$} with respect to the dimension function $f$ is then defined as \[
  \haus[f]{F} := \lim_{\rho \to 0} \Haus[f]_\rho(F).
\]

When $f(r) = r^s$, for some $s \ge 0$, the measure $\Haus[f]$ is the familiar $s$-dimensional Hausdorff measure, which we denote by $\Haus[s]$.
The \emph{Hausdorff dimension}, $\hdim{F}$ of a set $F$ is defined as \[
  \hdim{F} := \inf \setofall s \ge 0 \suchthat \haus[s]{F} = 0 \setend.
\]

With these definition in place, we recall
\begin{theorem}[The Mass Transference Principle, {\cite[Theorem 3]{mtp}}]\label{thm:mtp}
  Let $(X, d)$ be as above and let $\set{B_i}_{i \in \N}$ be a sequence of balls in $X$ with $\mathrm{rad}(B_i) \to 0$ as $i \to \infty$.
  Let $f$ be dimension function and $g$ be a doubling dimension function such that $\frac{f(x)}{g(x)}$ is monotonic and suppose that for any ball $B \in X$ \[
    \haus[g]{B \cap \limsup_{i \to \infty} B_i^f} = \haus[g]{B}.
  \]
  Then for any ball $B$ in $X$ \[
    \haus{B \cap \limsup_{i \to \infty} B_i^g} = \haus[f]{B}.
  \]
\end{theorem}

\subsection{Proof of Theorem \ref{thm:jarnik}}

The proof of Theorem \ref{thm:jarnik} naturally comes in two parts. Either $p_2$ is prime, or infinite; that is, we are either approximating in a $p$-adic space or the Euclidean space, respectively. We take care of both in one fell swoop.

In both cases, the convergence cases follows by an application of a covering argument and using the the Hausdorff-Cantelli Lemma; as in the classical case with the Borel--Cantelli Lemma.
The divergence case relies on the Mass Transference Principle and our Khintchine's Theorem above. To do so, we must first reformulate the problem so as to apply Theorem \ref{thm:p_to_p}.

\begin{proof}
  Suppose $f$ is a dimension function and let $B_{p_1} \subseteq \Zp[p_1]$ be given; here we take $\Zp[p_1]$ to be either the $p$-adic integers or, $\RmodZ$, depending on if $p_1$ is prime or infinite, respectively.
  
  Recall that we can express the $\psi$-approximables as a $\limsup$ set \[
    \pappby[p_2]{\psi}{B_{p_1}} = \limsup_{n \to \infty} (U_n \cup V_n) = \bigcap_{N = 1}^\infty \bigcup_{n \ge N} \parens*{U_n \cup V_n}
  \]
  where
  \begin{equation}\label{eq:limsup_sets}
    U_n = \bigcup_{\substack{a = 1 \\ (a,n) = 1 \\ \frac{a}{n} \in B_{p_1}}}^n B_{p_2}\parens*{ \frac{a}{n}, \psi(n)} \qquad \text{and} \qquad
    V_n = \bigcup_{\substack{a = 1 \\ (a,n) = 1 \\ \frac{n}{a} \in B_{p_1}}}^n B_{p_2}\parens*{ \frac{n}{a}, \psi(n)}.
  \end{equation}
  
  To prove the convergence part, let $\rho > 0$ and let $N(\rho) \in \N$ by such that $\psi(n) < \rho$, for all $n \ge N(\rho)$.
  Then, \[
    \bigcup_{n \ge N} \parens*{U_n \cup V_n} = \bigcup_{n \ge N} \bigcup_{\substack{a = 1 \\ (a,n) = 1 \\ \frac{a}{n} \in B_{p_1}}}^n B_{p_2}\parens*{ \frac{a}{n}, \psi(n)} \cup
    \bigcup_{\substack{a = 1 \\ (a,n) = 1 \\ \frac{n}{a} \in B_{p_1}}}^n B_{p_2}\parens*{ \frac{n}{a}, \psi(n)}
  \]
  is a $\rho$-cover of $\pappby[p_2]{\psi}{B_{p_1}}$.
  We then have
  \begin{align*}
    \Haus[f]_\rho\parens*{\pappby[p_2]{\psi}{B_{p_1} } }
      &\le \sum_{n \ge N} \parens*{ \sum_{\substack{a = 1 \\ (a,n) = 1 \\ \frac{a}{n} \in B_{p_1}}}^n V^f\parens*{ B_{p_2}\parens*{ \frac{a}{n}, \psi(n)} } +
      \sum_{\substack{a = 1 \\ (a,n) = 1 \\ \frac{n}{a} \in B_{p_1}}}^n V^f\parens*{  B_{p_2}\parens*{ \frac{n}{a}, \psi(n)} } } \\
      &\le \sum_{n \ge N} 2 n f(\psi(n)). \stepcounter{equation}\tag{\theequation}\label{eq:rho-cover}
  \end{align*}
  Since we have assumed $\sum\limits_{n \ge N} n f(\psi(n))$ converges, we can make the sum on the right-hand side of \eqref{eq:rho-cover} as small as we like by choosing $\rho$ sufficiently small.
  From the definition of Hausdorff measure we have $\haus[f]{\pappby[p_2]{\psi}{B_{p_1} }} = 0$.
  
  We now proceed to show the divergence case.
  For the sake of notation, enumerate the sets in \eqref{eq:limsup_sets} into a single sequence, $\set{A_i}$.
  Given a ball $A_i = B_{p_2}\parens*{ \frac{a}{n}, \psi(n)}$, we have \[
    A^f_i := B_{p_2}\parens*{ \frac{a}{n}, f( \psi(n)) }.
  \]
  Observe that \[
    \pappby[p_2]{\theta}{B_{p_1}} = \limsup_{i \to \infty} A^f_i,
  \]
  where $\theta = f \circ \psi$.
  Then, in the divergence case, we have \[
    \sum_{n = 1}^\infty n \theta(n) = \sum_{n = 1}^\infty n f(\psi(n)) = \infty.
  \]
  One may easily check that $\theta$ satisfies the regularity condition for Theorem \ref{thm:p_to_p} since $f$ is nondecreasing and doubling and $\psi$ already satisfies the required regularity condition.
  On applying Theorem \ref{thm:p_to_p}, we have \[
    \m[p_2]{\pappby[p_2]{\theta}{B_{p_1}}} = 1.
  \]
    
  Moreover, it is known that $\haus[1]{\cdot} = \m[p_2]{\cdot}$ for $p_2 = \infty$ (i.e. in $\RmodZ$). This follows from our use of the \emph{radius}, as opposed to diameter, in our definition of Hausdorff Measure.
  Similarly, when $p_2$ is a prime, we have $\haus[1]{\cdot} = \m[p_2]{\cdot}$.
  By comparing the definitions of Haar measure in $\Qp$ and the definition of $\Haus[g]$ above, we immediately have equality of the two measures.
  
  Returning to our proof, we have \[
    \haus[1]{\pappby[p_2]{\theta}{B_{p_1}}} = 
    \m[p_2]{\pappby[p_2]{\theta}{B_{p_1}}} = 1
  \]
  Which in turn implies
  \[
    \haus[1]{\Zp[p_2] \cap \limsup_{i \to \infty} A^f_i} = \haus[1]{\pappby[p_2]{\theta}{B_{p_1}}} = 1 = \haus[1]{\Zp[p_2]}.
  \]
  So, we may apply the Mass Transference Principle with $B = \Zp[p_2]$ and $B_i = A_i$ to get
  \begin{align*} 
    \haus[f]{\pappby[p_2]{\psi}{B_{p_1}}}
      &= \haus{\Zp[p_2] \cap \limsup_{i \to \infty} A_i}
      = \haus{\Zp[p_2]},
  \end{align*}
  which concludes the proof.
\end{proof}

\subsection{Proof of Corollary \ref{cor:jarnik-besicovitch}}

The result follows from applying Theorem \ref{thm:jarnik} with $\psi(q) = q^{-\tau}$ and $f(r) = r^s$.
We include the details below.

\begin{proof}
  Recall that, by definition, we have \[
    \hdim{\pappby[p_2]{\tau}{B_{p_1}}} := \inf \setofall s \ge 0 \suchthat \haus[s]{\pappby[p_2]{\tau}{B_{p_1}}} = 0 \setend.
  \]
  On applying Theorem \ref{thm:jarnik} we have
  \begin{equation}\label{eq:jarnik-besicovitch-1}
    \haus[s]{\pappby[p_2]{\tau}{B_{p_1}}} = \begin{cases}
      \haus[s]{\Zp[p_2]} & \sum_{n=1}^\infty n \parens*{n^{-\tau}}^{s} = \infty \\
      0 & \sum_{n=1}^\infty n \parens*{n^{-\tau}}^{s}  < \infty.
    \end{cases}
  \end{equation}
  Note that this can be rephrased as investigating the convergence/divergence of the series \[
    \sum_{n=1}^\infty n^{1 - s\tau}.
  \]
  We know that this series converges if and only if $1 - s \tau < -1$; Equivalently, the series converges if and only if $s > \frac{2}{\tau}$.
  From this, we can rewrite \eqref{eq:jarnik-besicovitch-1} as 
  \[
    \haus[s]{\pappby[p_2]{\tau}{B_{p_1}}} = \begin{cases}
      \haus[s]{\Zp[p_2]} & s \le \frac{2}{\tau} \\
      0 & s > \frac{2}{\tau}.
    \end{cases}
  \]
  Which in turn implies $\hdim{\pappby[p_2]{\tau}{B_{p_1}}} = \frac{2}{\tau}$, as $\haus[s]{\Zp[p_2]} = \infty$ for $s < 1$.
  Note that this last step uses our requirement that $\tau > 2$. In particular, since $\tau > 2$, we have $\frac{2}{\tau} < 1$.
\end{proof}

\subsection*{Acknowledgments}

The author would like to thank his advisor, Felipe Ram\'irez, for his mentorship and ever-present excitement through the many avenues this project has taken.
He would also like to thank Demi Allen for helpful conversations as well as Victor Beresnevich, Alan Haynes, and Matthew Palmer for helpful comments on early preprints of this paper.
Lastly he would like to thank the anonymous referee for corrections and helpful suggestions that greatly improved this paper.

\bibliographystyle{alpha}

\bibliography{draft.bib}

\end{document}